\documentclass[french,final,3p,times]{amsart}
\usepackage[T1]{fontenc}
\usepackage[latin9]{inputenc}
\usepackage{color}
\usepackage{babel}
\usepackage[normalem]{ulem}
\usepackage{float}
\usepackage{amsthm}
\usepackage{amstext}
\usepackage{amssymb}
\usepackage{graphicx}
\usepackage{esint}
\usepackage[numbers]{natbib}
\usepackage[unicode=true,
 bookmarks=true,bookmarksnumbered=false,bookmarksopen=false,
 breaklinks=false,pdfborder={0 0 1},backref=false,colorlinks=true]
 {hyperref}
\hypersetup{
 colorlinks=true,    linkcolor=blue,    filecolor=blue,  urlcolor=blue,    linkcolor=blue, citecolor=blue,colorlinks=blue,bookmarks=blue}

\makeatletter
\numberwithin{equation}{section}
\numberwithin{figure}{section}
\theoremstyle{plain}
\newtheorem{thm}{\protect\theoremname}
\theoremstyle{definition}
\newtheorem{defn}[thm]{\protect\definitionname}
\ifx\proof\undefined
\newenvironment{proof}[1][\protect\proofname]{\par
\normalfont\topsep6\p@\@plus6\p@\relax
\trivlist
\itemindent\parindent
\item[\hskip\labelsep\scshape #1]\ignorespaces
}{%
\endtrivlist\@endpefalse
}
\providecommand{\proofname}{Proof}
\fi
\theoremstyle{plain}

\theoremstyle{plain}
\newtheorem{lem}[thm]{\protect\lemmaname}
\theoremstyle{plain}
\newtheorem{cor}[thm]{\protect\corollaryname}

\title[\hfil Robustesse de la stabilité globale pour les syst\`emes perturbés] 
{Robustesse de la stabilité globale asymptotique des équilibres pour les syst\`emes dynamiques dissipatifs perturbés}

\author[Mohammed Amine Hamra]
{Mohammed Amine Hamra}

\thanks{M. Amine Hamra \newline Laboratoire Syst\`emes Dynamiques et Applications, Université de Tlemcen, Algérie.\newline \emph{E-mail address} : {\color{blue}{ammathine@gmail.com}}}


\date{06-01-2015}

\makeatother

\usepackage{babel}
\makeatletter
\addto\extrasfrench{%
   \providecommand{\fg}{\ifdim\lastskip>\z@\unskip\fi~\frqq}%
}

\makeatother
  \providecommand{\corollaryname}{Corollaire}
  \providecommand{\definitionname}{Définition}
  \providecommand{\lemmaname}{Lemme}
\providecommand{\theoremname}{Théorème}

\begin{document}

\begin{abstract}
Nous présentons dans ce papier des résultats théoriques pour la théorie
des perturbations régulières des systèmes dynamiques dissipatifs.
Plus précisément, nous voulons démontrer que la stabilité d'un équilibre
globalement asymptotiquement stable, persiste sous de petites perturbations
du champ de vecteurs. 
\end{abstract}

\maketitle
\noindent Considérons le système dynamique suivant
\[
\varDelta_{\epsilon}:\,\stackrel{.}{x}\,=f(x(t),\epsilon),
\]
où $f\,\,:\,\,(\varOmega\subset\mathbb{R}^{n})\times[-\epsilon_{0},\epsilon_{0}]\rightarrow\mathbb{R}^{n}$
sera supposée continûment différentiable par rapport à $x$; et continue
en $\epsilon$, elle vérifie donc les hypothèses du théorème de Cauchy-Lipschitz
qui assure l'existence et l\textquoteright unicité d'une solution
maximale $x_{\epsilon}(x_{0},t)$ de $\varDelta_{\epsilon}$ vérifiant
$x_{\epsilon}(x_{0},0)=x_{0}$. Supposons que le système non perturbé
\[
\varDelta_{0}:\,\stackrel{.}{x}\,=f(x(t),0)
\]
admet un point d\textquoteright équilibre $x^{\ast}$ globalement
asymptotiquement stable. Sous des hypothèses de dissipativité uniforme
par rapport au paramètre de perturbation, le théorème \ref{thm:globalstab}
montre que l'équilibre $x^{\ast}$ persiste sous de petites perturbations
du champ de vecteurs $\varDelta_{0}$, plus exactement, pour $\epsilon$
suffisamment petit, il existe un équilibre $x_{\epsilon}^{\ast}$
globalement asymptotiquement stable et proche de $x^{\ast}$.

Rappelons quelques définitions et théorèmes de base des systèmes dynamiques
utiles pour la suite.
\begin{defn}
Un champ de vecteurs linéaire $L\in\mathcal{L}(\mathbb{R}^{n})$ est
hyperbolique si aucune valeur propre ne rencontre l\textquoteright axe
imaginaire. L'indice de $L$ est le nombre de ses valeurs propres
à partie réelle négative.
\end{defn}
\begin{lem}[{\cite[Proposition 2.18]{Palis1982}}]
\label{lemme index}Si $L\in\mathcal{L}(\mathbb{R}^{n})$ est un
champ de vecteurs hyperbolique alors il existe un voisinage $\mathcal{V}\subset\mathcal{L}(\mathbb{R}^{n})$
de $L$ tel que tous les $T\in\mathcal{V}$ ont le même indice que
$L$.
\end{lem}

\begin{defn}
Soit $\varOmega$ un ouvert de $\mathbb{R}^{n}$. Un ensemble des
conditions initiales $K_{0}\subset\varOmega$ est dit absorbant pour
le système dynamique $\stackrel{.}{y}=f(y)$ \textit{\emph{défini
sur $\varOmega$ }}si pour toute partie bornée $K$ de $\varOmega$
il existe un temps fini $t_{0}=t_{0}(K)$ tel que $y(t,K)\subset K_{0}$
pour tout $t{\rm \geqslant}t_{0}$.
\end{defn}

\begin{defn}
\textit{\emph{Un ensemble }}\emph{${\rm \mathfrak{A}}$ }\textit{\emph{est
un attracteur}}\emph{ }\textit{\emph{global d'un }}système dynamique\emph{
$\stackrel{.}{y}=f(y)$ }\textit{\emph{défini sur $\varOmega$ si}}\emph{
${\rm \mathfrak{A}}$ }\textit{\emph{est une partie compacte, invariante,
et }}\emph{$\omega(B)\subset{\rm \mathfrak{A}}$ }\textit{\emph{pour
tout ensemble borné $B$. }}
\end{defn}
De cette définition, il est clair que l'attracteur global contient
tous les ensembles limites, et si un attracteur global existe, il
est unique.
\begin{thm}[{\cite[Theorem 1.1]{Temam1988}}]
\label{thm:existence attracteur}Soit $\stackrel{.}{y}=f(y)$ un
système dynamique défini sur un ouvert $\varOmega\subset\mathbb{R}^{n}.$
On suppose que ce système admet un ensemble absorbant borné $K_{0}\subset\varOmega$,
alors $\omega(K_{0})$ est l'unique attracteur global dans $\varOmega$.
\end{thm}
On désigne par $d(x,\mathbf{A})=\inf\{d(x,y)$: $y\in\mathbf{A}\}$
la distance de $x\in\mathbb{R}^{n}$ à l\textquoteright ensemble $\mathbf{A}\subset\mathbb{R}^{n}$.
$B(x,r),\,\bar{B}(x,r)$ représentent respectivement la boule ouverte
et la boule fermée de centre $x$ et de rayon $r$ .

Nous pouvons maintenant énoncer le résultat suivant:
\begin{thm}
\label{thm:globalstab}Supposons que le système $\varDelta_{\epsilon}$
est dissipatif uniformément par rapport à $\epsilon$. Supposons que
le système non perturbé $\varDelta_{0}$ admet un point d\textquoteright équilibre
$x^{\ast}$ globalement asymptotiquement stable dans $\varOmega$
et localement exponentiellement stable. Alors il existe une constante
positive $\epsilon^{*}\leqslant\epsilon_{0}$ assez proche de $0$
telle que $\varDelta_{\epsilon}$ possède un équilibre unique $x^{\ast}(\epsilon)$,
globalement asymptotiquement stable dans $\varOmega$ pour tout $\epsilon\in[-\epsilon^{*},\epsilon^{*}]$.
\end{thm}
\begin{proof}
Puisque le système $\varDelta_{\epsilon}$ est uniformément dissipatif
par rapport à $\epsilon$, il possède un ensemble compact $\mathcal{K\subset}\varOmega$
uniformément absorbant. D\textquoteright après le Théorème \ref{thm:existence attracteur},
l'ensemble oméga limite $\omega_{\epsilon}(\mathcal{K})=\mathcal{A}_{\epsilon}\subset\mathcal{K}$
est l'unique attracteur global dans $\varOmega$ pour tout $\epsilon\in[-\epsilon_{0},\epsilon_{0}]$,
où $\omega_{\epsilon}(.)$ désigne les ensembles oméga limites relatifs
à $\varDelta_{\epsilon}$.

\noindent Soient $\eta>0$ fixé, $\bar{B}(x^{\ast},\eta/2)$ la boule
fermée de centre $x^{\ast}$ et de rayon $\eta/2$ et soit $x(x_{0},t)$
la solution de l\textquoteright équation $\varDelta_{0}$ qui passe
par $x_{0}$ au temps $t=0$. Le point d'équilibre $x^{\ast}$ est
globalement asymptotiquement stable pour $\varDelta_{0}$ dans $\varOmega$,
donc il existe un temps fini $T_{\eta}$ tel 

\noindent 
\begin{equation}
x(\varOmega,t)\subset\bar{B}(x^{\ast},\eta/2),\,\,\,\forall t\geqslant T_{\eta}.\label{eq:10}
\end{equation}
D\textquoteright autre part, par la dépendance continue de la solution
par rapport aux conditions initiales et aux paramètres (i.e la solution
$x_{\epsilon}(x_{0},t)$ est uniformément continue sur $\varOmega\times[-\epsilon_{0},\epsilon_{0}]$)
et par le fait que $T_{\eta}$ est indépendant de $\epsilon$, il
existe $\hat{\epsilon}(\eta)$ tel que pour tout $\epsilon\in[-\hat{\epsilon},\hat{\epsilon}]$

\noindent 
\begin{equation}
d(x_{\epsilon}(x_{0},T_{\eta}),\thinspace x(x_{0},T_{\eta}))<\eta/2,\,\,\forall x_{0}\in\mathcal{K}.\label{eq:11}
\end{equation}
En utilisant l'inégalité triangulaire, avec (\ref{eq:10}) et (\ref{eq:11}),
nous obtenons alors
\begin{alignat*}{1}
d(x_{\epsilon}(x_{0},T_{\eta}),\thinspace x^{\ast})\leqslant & \,d(x_{\epsilon}(x_{0},T_{\eta}),\thinspace x(x_{0},T_{\eta}))\\
 & +d(x(x_{0},T_{\eta}),\thinspace x^{\ast})<\eta.
\end{alignat*}
Donc,
\begin{equation}
x_{\epsilon}(\mathcal{A}_{\epsilon},T_{\eta})=\mathcal{A}_{\epsilon}\subset B(x^{\ast},\eta),\label{eq:57}
\end{equation}

\noindent pour tout $\epsilon\in[-\hat{\epsilon},\hat{\epsilon}]$,
puisque $\mathcal{A}_{\epsilon}$ est invariant.

\noindent D'autre part, puisque l'équilibre de $\varDelta_{0}$ est
localement exponentiellement stable, la matrice Jacobienne $\mathrm{D}f(x^{\ast},0)$
de $\varDelta_{0}$ au point d'équilibre $x^{\ast}$ est Hurwitz,
et donc, son rang est maximal. D'après le théorème des fonctions implicites
il existe, un voisinage ouvert $U$ de $x^{\ast}$, une constante
$\epsilon_{1}>0$ et une application $\epsilon\in[-\epsilon_{1},\epsilon_{1}]\mapsto x^{\ast}(\epsilon)\in U$
continue telle que $x^{\ast}(0)=x^{\ast}$ et $x^{\ast}(\epsilon)$
est un point d'équilibre pour $\varDelta_{\epsilon}$. Soit $J_{\epsilon}=\mathrm{D}f(x^{\ast}(\epsilon),\epsilon)$
la matrice Jacobienne de $\varDelta_{\epsilon}$ au point d'équilibre
$x^{\ast}(\epsilon)$. Le développement limité de $J_{\epsilon}$
au voisinage de $\epsilon=0$ nous donne
\[
J_{\epsilon}=J_{0}+O(\epsilon),
\]
où $J_{0}=\mathrm{D}f(x^{\ast}(0),0)=\mathrm{D}f(x^{\ast},0)$.

\noindent Comme le spectre de la matrice Jacobienne $J_{\epsilon}$
dépend continument de $\epsilon$, le Théorème des fonctions implicites
affirme qu'il existe $\epsilon_{2}>0$ tel que l'indice de $J_{\epsilon}$
est le même que l'indice de $J_{0}$ pour tout $\epsilon\in[-\epsilon_{2},\epsilon_{2}]$.
Donc, l'indice de l\textquoteright équilibre perturbé $x^{\ast}(\epsilon)$
est maximal. D\textquoteright après le théorème de Hartman-Grobman,
il existe un voisinage $B(x^{\ast}(\epsilon),r_{\epsilon})$ de $x^{\ast}(\epsilon)$
tel que toute courbe d'une solution de $\varDelta_{\epsilon}$ avec
$\left|\,\epsilon\,\right|\leqslant\min\{\epsilon_{0},\epsilon_{1},\epsilon_{2}\}$
partant de $B(x^{\ast}(\epsilon),r_{\epsilon})$ tend vers l'équilibre
hyperbolique $x^{\ast}(\epsilon)$. 

\noindent Par continuité de $x^{\ast}(\epsilon)$, il existe $\epsilon_{3}$
et une constante $r_{0}>0$, telle que 
\[
B(x^{\ast},r{}_{0})\subset\bigcap_{\left|\,\epsilon\,\right|<\epsilon_{3}}B(x^{\ast}(\epsilon),r_{\epsilon}),
\]

\noindent étant entendu que, d\textquoteright après l'attractivité
de l\textquoteright équilibre $x^{\ast}$, le minimum de $r_{\epsilon}$
est strictement positif. Par (\ref{eq:57}), il existe $\epsilon_{4}(r{}_{0})$
tel que pour tout $\epsilon\in[-\epsilon_{4},\epsilon_{4}]$, 
\[
\mathcal{A}_{\epsilon}\subset B(x^{\ast},r{}_{0}).
\]
Soit $\left|\,\epsilon\,\right|\leqslant\epsilon^{*}=\min\{\epsilon_{0},\epsilon_{1},\epsilon_{2},\epsilon_{3},\epsilon_{4}\}$.

Comme $\mathcal{A}_{\epsilon}\subset B(x^{\ast},r{}_{0})$, toute
trajectoire initialisée dans $\Omega$ entre dans $B(x^{\ast},r{}_{0})$
en temps fini. Mais il est clair que pour tout $\epsilon\in[-\epsilon^{*},\epsilon^{*}]$,
toute solution est dans une des boules $B(x^{\ast}(\epsilon),r_{\epsilon})$
et converge vers $x^{\ast}(\epsilon)$ lorsque $t$ tend vers l'infini.
ce qui achève la démonstration.
\end{proof}
Le corollaire suivant reprend les résultats du théorème dans le cas
des systèmes définis dans le cône positif avec une hypothèse supplémentaire
assurant que l'équilibre du problème perturbé ne tend pas vers le
bord quand $\epsilon$ tend vers zéro.
\begin{cor}
\label{cor:global positif}Supposons satisfaites les hypothèses du
théorème \emph{(\ref{thm:globalstab})} avec $\varOmega\subset\mathbb{R}_{+}^{n}$.
Si de plus, le système $\varDelta_{\epsilon}$ est uniformément persistant
et que l'équilibre $x^{\ast}$ du système non perturbé $\varDelta_{0}$
est \uline{positif,} alors il existe une constante positive $\epsilon^{*}\leqslant\epsilon_{0}$
assez proche de $0$ telle que $\varDelta_{\epsilon}$ possède un
équilibre unique $x^{\ast}(\epsilon)$, globalement asymptotiquement
stable dans $\varOmega$ qui soit positif, pour tout $\epsilon\in[-\epsilon^{*},\epsilon^{*}]$.
\end{cor}
Nous voulons maintenant montrer une limite uniforme de la solution
d\textquoteright un système dynamique dissipatif sur un intervalle
infini.
\begin{cor}
\label{theo:3.6.2}Sous les hypothèses du théorème \textup{\ref{thm:globalstab}},
il existe $\epsilon^{**}\in[0,\epsilon_{0}]$ tel que pour tout $\eta>0$,
pour tout $\epsilon\in[-\epsilon^{**},\epsilon^{**}]$,
\[
d(x_{\epsilon}(x_{0},t),\thinspace x(x_{0},t))<\eta,
\]
 pour tout $t\geqslant0$.
\end{cor}
\begin{proof}
Soit $\eta>0$ fixé. Puisque $x^{\ast}$ est un équilibre globalement
asymptotiquement stable pour $\varDelta_{0}$, il existe un temps
fini $\theta_{1}(\eta)$ tel que
\[
d(x(x_{0},t),\thinspace x^{\ast})\leqslant\frac{\eta}{3}\mbox{ pour tout }t\geqslant\theta_{1}(\eta).
\]
D\textquoteright après le Théorème \ref{thm:globalstab} il existe
un équilibre $x^{\ast}(\epsilon)$, $O(\epsilon)$-proche de $x^{\ast}$,
et donc il existe $\epsilon_{1}<\epsilon_{0}$ tel que pour tout $\epsilon\in[-\epsilon_{1},\epsilon_{1}]$ 

\[
d(x^{\ast}(\epsilon),\thinspace x^{\ast})\leqslant\frac{\eta}{3}.
\]
Le même théorème montre qu'il existe $\epsilon^{*}>0$ tel que $x^{\ast}(\epsilon)$
est globalement asymptotiquement stable pour $\varDelta_{\epsilon}$.
Il en résulte qu'il existe un temps fini $\theta_{2}(\eta)>0$ tel
que

\[
d(x_{\epsilon}(x_{0},t),\thinspace x^{\ast}(\epsilon))\leqslant\frac{\eta}{3}\mbox{ \,pour tous }t\geqslant\theta_{2}\mbox{ et }-\epsilon^{*}<\epsilon<\epsilon^{*}.
\]

\noindent Soit $T=\max\{\theta_{1},\theta_{2}\}$. Alors,

\noindent 
\begin{alignat*}{1}
d(x_{\epsilon}(x_{0},t),\thinspace x(x_{0},t))\leqslant & \,d(x_{\epsilon}(x_{0},t),\thinspace x^{\ast}(\epsilon))+d(x^{\ast}(\epsilon),\thinspace x^{\ast})+d(x(x_{0},t),\thinspace x^{\ast})\\
 & \leqslant\frac{\eta}{3}+\frac{\eta}{3}+\frac{\eta}{3}=\eta\,,\mbox{ }\forall t\geqslant T\mbox{ et }\left|\,\epsilon\,\right|<\min\{\epsilon_{1},\epsilon^{*}\}.
\end{alignat*}

\noindent Dans l'intervalle fini $[0,T]$, par la dépendance continue
de la solution par rapport aux paramètres \cite[Théorème 2, P 84]{Perko2013},
il existe $\epsilon_{2}<\epsilon_{0}$ tel que pour tout $\epsilon\in[-\epsilon_{2},\epsilon_{2}]$
\[
d(x_{\epsilon}(x_{0},t),\thinspace x(x_{0},t))<\eta.
\]
\end{proof}

\bibliographystyle{plain-fr}
\bibliography{references}

\end{document}